\documentclass[12pt]{amsart}
\usepackage{amsmath,amsthm,amsfonts,amssymb,latexsym,verbatim,bbm,longtable}
\input xy
\xyoption{all}
\usepackage{hyperref,multirow}
\usepackage{tikz, tabularx,booktabs,float}
\usepackage[shortlabels]{enumitem}
\usepackage{young}

\allowdisplaybreaks[2] \textwidth15.1cm \textheight22cm \headheight12pt \oddsidemargin.4cm
\theoremstyle{plain}
\newtheorem{theorem}{Theorem}[section]
\newtheorem{thm}[theorem]{Theorem}
\newtheorem{lemma}[theorem]{Lemma}
\newtheorem{proposition}[theorem]{Proposition}
\newtheorem{prop}[theorem]{Proposition}
\newtheorem{example}[theorem]{Example}
\newtheorem{definition}[theorem]{Definition}

\newtheorem{remark}[theorem]{Remark}

\newtheorem{cor}[theorem]{Corollary}
\newtheorem{conjecture}[theorem]{Conjecture}

\newcommand{\C}{\mathbb{C}}

\newcommand{\K}{\mathbb{K}}

\DeclareMathOperator{\mfS}{\mathfrak{S}}

\DeclareMathOperator{\Gr}{Gr}

\DeclareMathOperator{\Fl}{F\ell}

\newcommand{\bull}{\bullet}

\DeclareMathOperator{\Span}{span}

\begin{document}

\title[Pattern Avoidance and Schubert Varieties]{Pattern avoidance and fiber bundle structures on Schubert varieties}

\author{Timothy Alland}
\email{tim.alland@okstate.edu}

\author{Edward Richmond}
\email{edward.richmond@okstate.edu}

\maketitle
\begin{abstract}

We give a permutation pattern avoidance criteria for determining when the projection map from the flag variety to a Grassmannian induces a fiber bundle structure on a Schubert variety.   In particular, we introduce the notion of a split pattern and show that a Schubert variety has such a fiber bundle structure if and only if the corresponding permutation avoids the split patterns $3|12$ and $23|1$.  Continuing, we show that a Schubert variety is an iterated fiber bundle of Grassmannian Schubert varieties if and only if the corresponding permutation avoids (non-split) patterns 3412, 52341, and 635241.  This extends a combined result of Lakshmibai-Sandhya, Ryan, and Wolper who prove that Schubert varieties whose permutation avoids the ``smooth" patterns 3412 and 4231 are iterated fiber bundles of smooth Grassmannian Schubert varieties.


\end{abstract}
\section{Introduction}

Let $\K$ be an algebraically closed field and let
$$\Fl(n):=\{V_\bull=V_1\subset V_2 \subset\cdots \subset V_{n-1}\subset\K^n \ |\ \dim(V_i)=i\}$$
denote the complete flag variety on $\K^n.$  For each $r\in\{1,\ldots,n-1\}$, let
$$\Gr(r,n):=\{V\subset \K^n\ |\ \dim(V)=r\}$$ denote the Grassmannian of $r$-dimensional subspaces of $\K^n$ and consider the natural projection map
\begin{equation}\label{Eq:Projection map}
\pi_r:\Fl(n)\twoheadrightarrow \Gr(r,n)
\end{equation}
given by $\pi_r(V_\bull)= V_r.$  It is easy to see that the projection $\pi_r$ is a fiber bundle on $\Fl(n)$ with fibers isomorphic to $\Fl(r)\times\Fl(n-r)$.  The goal of this paper is to give a pattern avoidance criteria for when the map $\pi_r$ restricted to a Schubert variety of $\Fl(n)$ is also a fiber bundle.

\smallskip

Fix a basis $\{e_1,\ldots,e_n\}$ of $\K^n$ and let $E_i:=\Span\langle e_1,\ldots,e_i\rangle$.  Each permutation $w=w(1)\cdots w(n)$ of the symmetric group $W:=\mfS_n$ defines a Schubert variety
$$X_w:=\{V_\bull\in\Fl(n)\ |\ \dim(E_i\cap V_j)\geq r_w[i,j]\}$$
where $r_w[i,j]:=\#\{k\leq j \ |\ w(k)\leq i\}$.

\begin{thm}\label{T:main}
Let $r<n$ and $w\in W$. The projection $\pi_r$ restricted to $X_w$ is a Zariski-locally trivial fiber bundle if and only if $w$ avoids the split patterns $3|12$ and $23|1$ with respect to position $r$.
\end{thm}

If a permutation avoids a split pattern with respect to every position $r<n$, then that permutation avoids the pattern in the classical sense.  For a precise definition of split pattern avoidance, see Definition \ref{D:split patterns}.  Pattern avoidance has been used to combinatorially describe many geometric properties of Schubert varieties.  Most notably, Lakshmibai and Sandhya prove that a Schubert variety $X_w$ is smooth if and only if $w$ avoids the patterns 3412 and 4231 \cite{LS90}.  Pattern avoidance has also been used to determine when Schubert varieties are defined by inclusions, Gorenstein, factorial, have small resolutions and are local complete intersections \cite{BW03,De90,GR02,BMB07,WY06,UW13}.  For a survey of these results and many others see \cite{AB14}.

\subsection{Complete parabolic bundle structures}

For any positive integer $n$, define the set $[n]:=\{1,\ldots,n\}$.  The varieties $\Fl(n)$ and $\Gr(r,n)$ are the extreme examples in the collection of partial flag varieties on $\K^n$.  For any subset $\textbf{a}:=\{a_1<\cdots <a_k\}\subseteq[n-1]$, define the partial flag variety
$$\Fl(\textbf{a},n):=\{V^{\textbf{a}}_\bull:=V_{a_1}\subset V_{a_2} \subset\cdots \subset V_{a_k}\subseteq\K^n \ |\ \dim(V_{a_i})=a_i\}.$$
If $\textbf{b}\subseteq\textbf{a}$, then there is a natural projection map
$$\pi^{\textbf{a}}_{\textbf{b}}:\Fl(\textbf{a},n)\twoheadrightarrow\Fl(\textbf{b},n)$$ given by $\pi^{\textbf{a}}_{\textbf{b}}(V^{\textbf{a}}_\bull)=V^{\textbf{b}}_\bull$.  In other words, the map $\pi^{\textbf{a}}_{\textbf{b}}$ forgets the components of $V^{\textbf{a}}_\bull$ indexed by $\textbf{a}\setminus\textbf{b}$.  Note that the map $\pi_r=\pi^{[n-1]}_{\{r\}}$ from Equation \ref{Eq:Projection map}.  Any permutation $\sigma=\sigma(1)\cdots\sigma(n-1)\in\mfS_{n-1}$ defines a collection of nested subsets
$$\sigma_1\subset\sigma_2\subset\cdots\subset\sigma_{n-2}\subset \sigma_{n-1}=[n-1]\quad \text{where}\quad \sigma_i:=\{\sigma(1),\ldots,\sigma(i)\}.$$
The maps $\pi^{\sigma_{i}}_{\sigma_{i-1}}$ induce an iterated fiber bundle structure on the complete flag variety
\begin{equation}\label{Eq:Flag_bundle_map}
\Fl(n)\overset{\pi^{[n-1]}_{\sigma_{n-2}}}{\twoheadrightarrow} \Fl(\sigma_{n-2},n)\overset{\pi^{\sigma_{n-2}}_{\sigma_{n-3}}}{\twoheadrightarrow}\cdots\overset{\pi^{\sigma_{3}}_{\sigma_{2}}}{\twoheadrightarrow} \Fl(\sigma_2,n)\overset{\pi^{\sigma_{2}}_{\sigma_{1}}}{\twoheadrightarrow}
\Fl(\sigma_1,n)\twoheadrightarrow pt
\end{equation}
where the fibers of each map $\pi^{\sigma_{i}}_{\sigma_{i-1}}$ are isomorphic to Grassmannians.
\begin{definition}\label{D:parabolic bundle}
Let $w\in W$.  We say $X_w$ has a \textbf{complete parabolic bundle structure} if there is a permutation $\sigma\in\mfS_{n-1}$ such that the maps $\pi^{\sigma_{i}}_{\sigma_{i-1}}$ induce an iterated fiber bundle structure on the Schubert variety
\begin{equation}\label{Eq:Flag_bundle_map_Schubert}
X_w=X_{n-1}\overset{\pi^{[n-1]}_{\sigma_{n-2}}}{\twoheadrightarrow} X_{n-2}\overset{\pi^{\sigma_{n-2}}_{\sigma_{n-3}}}{\twoheadrightarrow}\cdots\overset{\pi^{\sigma_{3}}_{\sigma_{2}}}{\twoheadrightarrow} X_2\overset{\pi^{\sigma_{2}}_{\sigma_{1}}}{\twoheadrightarrow}
X_1\twoheadrightarrow pt
\end{equation}
where $X_i:=\pi^{[n-1]}_{\sigma_{i}}(X_n)\subseteq\Fl(\sigma_i,n)$.  In other words, each map $$\pi^{\sigma_{i}}_{\sigma_{i-1}}: X_i\twoheadrightarrow X_{i-1}$$ is a Zariski-locally trivial fiber bundle.
\end{definition}
Not every Schubert variety has a complete parabolic bundle structure.  The smallest such Schubert variety is $X_{3412}$ (See Example \ref{Ex:3}).  When $\K=\C$, Ryan showed that any smooth Schubert variety has complete parabolic bundle structure \cite{Ry87}.  Wolper later generalized this result to include Schubert varieties over any algebraically closed field \cite{Wo89}.  Combining these results with the Lakshmibai-Sandhya smoothness criteria, we have:
\begin{thm}\label{T:Ryan_LS}\emph{(\cite{Ry87,Wo89,LS90})}
If $w$ avoids patterns $3412$ and $4231$, then $X_w$ has a complete parabolic bundle structure.
\end{thm}
An analogous result to Theorem \ref{T:Ryan_LS} holds true for rationally smooth Schubert varieties of any finite type \cite{RS14}.  We remark that the converse of Theorem \ref{T:Ryan_LS} is false.  For example, the permutation $\sigma=213$ induces a complete parabolic bundle structure on $X_{4231}$ (See Example \ref{Ex:2}).  One application of Theorem \ref{T:main} is a pattern avoidance characterization of Schubert varieties that have complete parabolic bundle structures.

\begin{thm}\label{T:main2}
The permutation $w$ avoids patterns $3412$, $52341$ and $635241$ if and only if the Schubert variety $X_w$ has a complete parabolic bundle structure.
\end{thm}


The key property used to prove both Theorems \ref{T:main} and \ref{T:main2} is the notion of a Billey-Postnikov (BP) decomposition $w=vu$ of a permutation (see Proposition \ref{P:BP_def} for the definition).  The term BP decomposition was originally used in \cite{OY10} to describe a certain factorization condition on the Poncar\'{e} polynomials of $w,v,u$ observed by Billey and Postnikov in \cite{BP05}.  Since then, several equivalent conditions have been given to describe this property (see \cite[Section 4]{RS14}).


\subsection{Applications to enumeration}

In the unpublished work \cite{Ha92}, Haiman calculates the generating series for the number of smooth Schubert varieties in $\Fl(n)$ (or equivalently, the number of permutations avoiding $3412$ and $4231$) using complete parabolic bundle structures.  This generating series also appears in \cite{Bo98} and \cite{BMB07}.  Analogous parabolic bundle structures were used to calculate the generating series of smooth and rationally smooth Schubert varieties of any classical type by the second author and Slofstra in \cite{RS15}.  To facilitate this calculation, a new combinatorial structure called a staircase diagram over a graph was introduced.  One immediate corollary of Theorem \ref{T:main2} and \cite[Theorem 5.1]{RS15} is the following.

\begin{cor}\label{C:main2}
Permutations in $\mfS_n$ avoiding the patterns $3412$, $52341$ and $635241$ are in bijection with staircase diagrams over the line graph on $n-1$ vertices with nearly-maximal labellings.
\end{cor}

For the definition of staircase diagrams over a graph and nearly-maximal labellings, see \cite{RS15}.  It is likely that a generating function for the number of permutations in $\mfS_n$ avoiding the patterns $3412$, $52341$ and $635241$ can be calculated by enumerating these labelled staircase diagrams using similar techniques in \cite{RS15}.  The first few coefficients of the corresponding generated series are listed in Table \ref{TBL:count} and were calculated using SAGE.

\begin{table}[H]
    \begin{tabular}{cccccccccc}
        \toprule
        $n=1$ & $n=2$ & $n=3$ & $n=4$ & $n=5$ & $n=6$&$n=7$ &$n=8$ &$n=9$ &$n=10$ \\
        \midrule
       1 &  2     & 6     & 23    &   102 & 492 & 2492 &13008 &69267& 374019  \\

        \midrule
    \end{tabular}
    \caption{Number of permutations in $\mfS_n$ avoiding $3412$, $52341$ and $635241$.}
    \label{TBL:count}
\end{table}

\subsection*{Acknowledgements}  The first author was partially supported by Oklahoma State University's Koslow Undergraduate Math Research Experience Scholarship. The second author was partially supported by a NSA Young Investigator's Grant and an Oklahoma State University Dean's Incentive Grant.  The program SAGE was used to collect data on BP decompositions of permutations in relation to pattern avoidance.

\section{Preliminaries}

For any integers $m<n$, define the interval $[m,n]:=\{m,m+1,\ldots,n\}$.  If $m=1$, then we set $[n]:=[1,n]$.  The symmetric group $W:=\mfS_n$ has simple generators $s_1,\ldots, s_{n-1}$ with Coxeter relations
$$s_i^2=1,\quad s_is_{i+1}s_i=s_{i+1}s_i s_{i+1},\quad\text{and}\quad s_is_j=s_js_i\quad\text{if}\ |i-j|>1.$$
We will often view elements $w\in W$ as permutations $w:[n]\rightarrow [n]$ where $s_i$ is identified with the simple transposition $(i,i+1)$.  In one-line notation, we have $w=w(1)w(2)\cdots w(n)$.  Diagrammatically, we represent $w$ with a permutation array with nodes marking the points $(w(i),i)$ using the convention that $(1,1)$ marks the upper left corner.  For example, $w=436125$ corresponds to the array:

$$
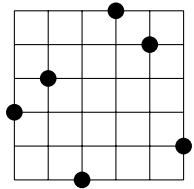
\begin{tikzpicture}[scale=0.45]
\draw[step=1.0,black] (0,0) grid (5,5);
\fill (0,2) circle (7pt);
\fill (1,3) circle (7pt);
\fill (2,0) circle (7pt);
\fill (3,5) circle (7pt);
\fill (4,4) circle (7pt);
\fill (5,1) circle (7pt);
\end{tikzpicture}$$

A \textbf{split pattern} $w=w_1|w_2\in W$ is a divided permutation where $w_1=w(1)\cdots w(j)$ and $w_2=w(j+1)\cdots w(n)$ for some $j\in[n-1]$.  We use split patterns to make the following modified definition of pattern containment and avoidance.

\begin{definition}\label{D:split patterns}
Let $k,r\leq n$.  Let $w=w(1)\cdots w(n)$ and $u=u(1)\cdots u(j)|u(j+1)\cdots u(k)$.  We say $w$ \textbf{contains the split pattern} $u$ with respect to position $r$ if there exists a sequence $(i_1<\cdots<i_k)\subseteq[n]$ such that
\begin{enumerate}
\item $w(i_1)\cdots w(i_k)$ has the same relative order as $u$
\item $i_j\leq r<i_{j+1}$.
\end{enumerate}
If $w$ does not contain $u$ with respect to position $r$, then we say $w$ \textbf{avoids the split pattern} $u$ with respect to position $r$.
\end{definition}

\begin{example}
Let $w=426135$ and $u=34|12$.  Then $w$ contains the split pattern $u$ with respect to position $r=3$, but avoids the split pattern $u$ with respect to all other positions.
$$
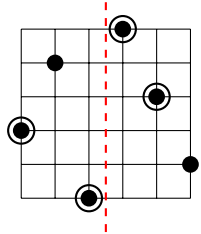
\begin{tikzpicture}[scale=0.45]
\draw[step=1.0,black] (0,0) grid (5,5);
\fill (0,2) circle (7pt);\draw[thick] (0,2) circle (11pt);
\fill (1,4) circle (7pt);
\fill (2,0) circle (7pt);\draw[thick] (2,0) circle (11pt);
\fill (3,5) circle (7pt);\draw[thick] (3,5) circle (11pt);
\fill (4,3) circle (7pt);\draw[thick] (4,3) circle (11pt);
\fill (5,1) circle (7pt);
\draw[dashed,thick, red] (2.5,-1)--(2.5,6);
\end{tikzpicture}$$
\end{example}

Note that part (1) of Definition \ref{D:split patterns} is the usual definition of pattern containment.  It is easy to see that if $w$ avoids a split pattern $u$ with respect to all $r\in[n-1]$, then $w$ avoids the non-split pattern $u$ in the usual sense.

We now go over some notation and properties of $W$ as a Coxeter group.  Let $S=\{s_1,\ldots,s_{n-1}\}$ denote the set of simple generators of $W$.  An expression of $w=s_{i_1}\cdots s_{i_k}$ is said to be \textbf{reduced} if $w$ cannot be written in a fewer number of simple generators.  The number of generators used in any reduced expression is called the \textbf{length} of $w$ and is denoted by $\ell(w)$.  If there are reduced expressions $w=s_{i_1}\cdots s_{i_{\ell(w)}}$ and $u=s_{j_1}\cdots s_{j_{\ell(u)}}$ where $(i_1,\ldots,i_{\ell(w)})$ is a subsequence of $(j_1,\ldots,j_{\ell(u)})$, then we say $w\leq u$ in the \textbf{Bruhat partial order} on $W$.  For any $w\in W$, define
\begin{align*}
S(w)&:=\{s\in S \ |\ s\leq w\}\\
D_L(w)&:=\{s\in S \ |\ \ell(sw)<\ell(w)\}\\
D_R(w)&:=\{s\in S \ |\ \ell(ws)<\ell(w)\}
\end{align*}
to be the \textbf{support} and \textbf{left and right descent sets} of $w$ respectively.  For any subset $J\subseteq S$, let $W_J$ denote the parabolic subgroup generated by $J$ and let $W^J$ denote the set of minimal length coset representatives of $W/W_J$.  For each $w\in W$ and $J\subseteq S$, there is a unique \textbf{parabolic decomposition} $w=vu$ where $v\in W^J$ and $u\in W_J$.  In terms of the supports and descents sets, it is well known than $u\in W_J$ if and only if $S(u)\subseteq J$.  Furthermore, we have $v\in W^J$ if and only if $D_R(v)\cap J=\emptyset.$

The parabolic decomposition $w=vu$ with respect to $J=S\setminus\{s_r\}$ can be described explicitly in terms of split patterns.
\begin{lemma}\label{L:Parabolic perm_structure}
Let $w\in W$ and write $$w=w_1|w_2=w(1)\cdots w(r)|w(r+1)\cdots w(n).$$  Let $w=vu$ be the parabolic decomposition with respect to $J=S\setminus\{s_r\}$.  Then
\begin{enumerate}
\item $v=v_1|v_2$ where $v_1$ and $v_2$ respectively consist of the entries of $w_1$ and $w_2$ arranged in increasing order.
\item $u=u_1|u_2$ where $u_1$ and $u_2$ are respectively the unique permutations on $[1,r]$ and $[r+1,n]$ with relative orders of $w_1$ and $w_2$.
\end{enumerate}
\end{lemma}

\begin{proof}
The lemma follows from the fact that $D_R(v)\subseteq \{s_r\}$ and that $s_r\notin S(u)$.
\end{proof}

\begin{example}
Let $w=541|623$. If $w=vu$ is the parabolic decomposition with respect to $J=S\setminus\{s_3\}$, then $v=145|236$ and $u=321|645$.

$$
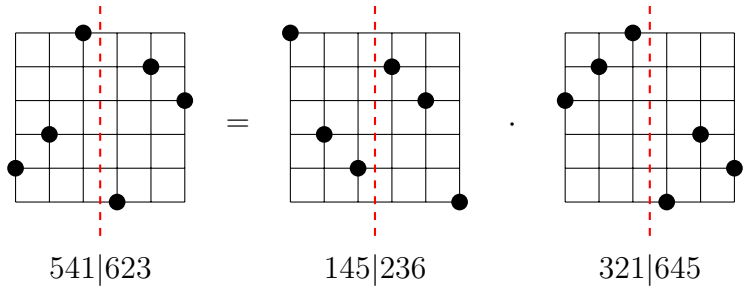
\begin{tikzpicture}[scale=0.45]
\draw[step=1.0,black] (0,0) grid (5,5);
\fill (0,1) circle (7pt);
\fill (1,2) circle (7pt);
\fill (2,5) circle (7pt);
\fill (3,0) circle (7pt);
\fill (4,4) circle (7pt);
\fill (5,3) circle (7pt);
\draw[dashed,thick, red] (2.5,-1)--(2.5,6);
\draw (2.5,-2) node {$541|623$};
\end{tikzpicture}
\begin{tikzpicture}[scale=0.45]
\draw (0,4.75) node {$=$};
\draw (-1,0) node {};\draw (1,0) node {};
\end{tikzpicture}
\begin{tikzpicture}[scale=0.45]
\draw[step=1.0,black] (0,0) grid (5,5);
\fill (0,5) circle (7pt);
\fill (1,2) circle (7pt);
\fill (2,1) circle (7pt);
\fill (3,4) circle (7pt);
\fill (4,3) circle (7pt);
\fill (5,0) circle (7pt);
\draw[dashed,thick, red] (2.5,-1)--(2.5,6);
\draw (2.5,-2) node {$145|236$};
\end{tikzpicture}
\begin{tikzpicture}[scale=0.45]
\draw (0,4.75) node {$\cdot$};
\draw (-1,0) node {};\draw (1,0) node {};
\end{tikzpicture}
\begin{tikzpicture}[scale=0.45]
\draw[step=1.0,black] (0,0) grid (5,5);
\fill (0,3) circle (7pt);
\fill (1,4) circle (7pt);
\fill (2,5) circle (7pt);
\fill (3,0) circle (7pt);
\fill (4,2) circle (7pt);
\fill (5,1) circle (7pt);
\draw[dashed,thick, red] (2.5,-1)--(2.5,6);
\draw (2.5,-2) node {$321|645$};
\end{tikzpicture}$$
\end{example}

In this case, when $J=S\setminus\{s_r\}$, each $v\in W^J$ corresponds to a unique Schubert variety in the Grassmannian $\Gr(r,n)$.  In particular, define the Schubert variety
$$X^J_v:=\{V\in \Gr(r,n)\ |\ \dim(V\cap E_j)\geq r_v[i,j]\}.$$  Geometrically, restricting $\pi_r$ to $X_w$ gives the projection
$$\pi_r:X_w\twoheadrightarrow X^J_v$$
where the generic fiber is isomorphic to the Schubert variety $X_u$.  We now give a combinatorial characterization for when $\pi_r$ is a fiber bundle.

\begin{proposition}\label{P:BP_def}\emph{(\cite[Theorem 3.3, Proposition 4.2]{RS14})}
Let $w\in W$ and $r<n$.  Let $w=vu$ be the parabolic decomposition with respect to $J=S(w)\setminus\{s_r\}$.  Then the following are equivalent.

\begin{enumerate}
\item $w=vu$ is a \textbf{BP decomposition} with respect to $J$.
\item $S(v)\cap J\subseteq D_L(u)$.
\item The projection $\pi_r:X_w\twoheadrightarrow X^J_v$ is a Zariski-locally trivial fiber bundle.
\end{enumerate}
\end{proposition}

The equivalencies in Proposition \ref{P:BP_def} are proved in \cite{RS14} and for this paper, we will take either parts (2) or (3) of Proposition \ref{P:BP_def} as the definition of BP decomposition (note that this definition corresponds to a ``Grassmannian BP decomposition" in \cite{RS14}).  The goal of Theorem \ref{T:main} is to give a pattern avoidance criteria on the permutation $w$ for any of these equivalent conditions.

Finally, we say $w$ has a \textbf{complete BP decomposition} if we can write
$$w=v_{k}\cdots v_1$$
where for every $i\in[k-1]$, we have $|S(v_i\cdots v_1)|=i$ and $v_{i}(v_{i-1}\cdots v_1)$ is a BP decomposition with respect to $S\setminus\{s_{r_i}\}$ where $s_{r_i}$ is the unique simple generator in $S(v_{i})\setminus S(v_{i-1}\cdots v_1)$.

Observe that the maps $\pi_r=\pi^{[n-1]}_{\{r\}}$ are not of the form $\pi^{\sigma_{i}}_{\sigma_{i-1}}$ used in Definition \ref{D:parabolic bundle}.  The next proposition gives the connection between BP decompositions and complete parabolic bundle structures on Schubert varieties.  The proposition follows directly from \cite[Lemma 4.3]{RS14} and the proof of \cite[Corollary 3.7]{RS14}.

\begin{proposition}\label{P:BP_complete}\emph{(\cite[Lemma 4.3, Corollary 3.7]{RS14})}
The permutation $w$ has a complete BP decomposition if and only if $X_w$ has a complete parabolic bundle structure.
\end{proposition}

The correspondence in Proposition \ref{P:BP_complete} can be made explicit.  Indeed, let $w=v_k\cdots v_1$ be a complete BP decomposition and let  $s_{r_i}$ denote the unique simple generator in $S(v_{i})\setminus S(v_{i-1}\cdots v_1)$.  Then any permutation $\sigma\in \mfS_{n-1}$ satisfying $\sigma(i)=r_{k-i-1}$ for $i\in[k]$ induces a complete parabolic bundle structure on $X_w$.

\begin{example}
Let $w=4231\in W$.  Then $w=(s_1s_3s_2)(s_1)(s_3)$ is a complete BP decomposition and $\sigma=213$ induces a complete parabolic bundle structure on $X_w$.  See Example \ref{Ex:2} for more details.
\end{example}

\section{Proof of Main theorems}

In this section we prove Theorems \ref{T:main} and \ref{T:main2}.  We begin with some important lemmas on permutations.

\begin{lemma}\label{L:support_v}
Let $v=v(1)\cdots v(n)\in W^J$ where $J=S\setminus\{s_r\}$.  Then $$S(v)=\{s_k\in S \ |\ v(r+1)\leq k< v(r)\}.$$
\end{lemma}

\begin{proof}
Note that $v\in W^J$ has at most one (right) descent occurring only at $(v(r),v(r+1))$.  Hence the simple transpositions in $S(v)$ are exactly those needed to transpose $(v(r),r)$ and $(v(r+1),r+1)$.
\end{proof}

\begin{lemma}\label{L:descent}
Let $u=u(1)\cdots u(n)\in W$.  Then $$D_L(u)=\{s_k\in S\ |\ u^{-1}(k+1)<u^{-1}(k)\}.$$
In other words, $s_k$ is a left descent of $u$ if and only if the node in the $k$-th row is to the right of the node in the $(k+1)$-th row in the permutation array of $u$.
\end{lemma}

\begin{proof}
The lemma follows from the fact that $D_L(u) = D_R(u^{-1})$ and if $A$ is the permutation array of $u$, then $A^{-1}= A^t$ is the permutation array of $u^{-1}$.
\end{proof}

\begin{example}
Take the permutation used before, $u = 436125$ which corresponds to the permutation array below.  In the first array, we mark the right descents and in the second array, we marks the left descents.
$$
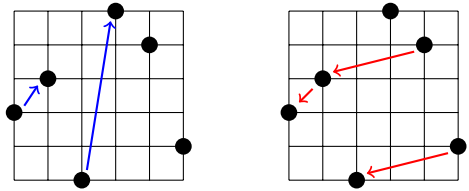
\begin{tikzpicture}[scale=0.45]
\draw[step=1.0,black] (0,0) grid (5,5);
\fill (0,2) circle (7pt);
\fill (1,3) circle (7pt);
\fill (2,0) circle (7pt);
\fill (3,5) circle (7pt);
\fill (4,4) circle (7pt);
\fill (5,1) circle (7pt);
\draw[thick, ->, blue] (2.15,.3)--(2.85,4.7);
\draw[thick, ->, blue] (.3,2.2)--(.7,2.8);
\end{tikzpicture}
\begin{tikzpicture}[scale=0.45]
\draw (0,4.75) node {};
\draw (-1,0) node {};\draw (1,0) node {};
\end{tikzpicture}
\begin{tikzpicture}[scale=0.45]
\draw[step=1.0,black] (0,0) grid (5,5);
\fill (0,2) circle (7pt);
\fill (1,3) circle (7pt);
\fill (2,0) circle (7pt);
\fill (3,5) circle (7pt);
\fill (4,4) circle (7pt);
\fill (5,1) circle (7pt);

\draw[thick, ->, red] (3.7,3.8)--(1.3,3.2);
\draw[thick, ->, red] (.7,2.7)--(.3,2.3);
\draw[thick, ->, red] (4.7,.8)--(2.3,.2);
\end{tikzpicture}$$
Hence, $D_R(u) = \{s_1,s_3\}$ and $D_L(u) = \{s_2,s_3,s_5\}$.
\end{example}

In the proofs of Theorem \ref{T:main} and \ref{T:main2}, we will often refer to sub-arrays or rectangular regions of a permutation array.  Let $A$ be the permutation array of $w=w(1)\cdots w(n)$.  We say a region $R$ of $A$ is \textbf{empty} if the interior of $R$ contains no nodes of the form $(w(i),i)$.  We say a region $R$ is \textbf{decreasing} if for every pair $(w(i),i), (w(j),j)$ in $R$, we have $i<j$ implies $w(i)>w(j)$.  Empty regions in a permutation array will be denoted by a shaded background and decreasing regions with be decorated (counter intuitively) with a northeast arrow.  Finally, we say a pair of nodes $(w(i),i), (w(j),j)$ are \textbf{increasing} if $i<j$ and $w(i)<w(j)$.

\begin{proof}[Proof of Theorem \ref{T:main}]
Fix $r<n$ and let $w\in W=\mfS_n$.  Let $w=vu$ be the parabolic decomposition with respect to $J=S\setminus\{s_r\}.$  By Proposition \ref{P:BP_def}, it suffices to prove that $w$ avoids the split patterns $3|12$ and $23|1$ with respect to position $r$ if and only if $S(v)\cap J=S(v)\setminus\{s_r\}\subseteq D_L(u)$.  Let
$$m:=\max\{w(k) \ |\ k\leq r\}\qquad \text{and}\qquad l:=\min\{w(k) \ |\ k>r\}.$$
The nodes $(m,w^{-1}(m))$ and $(l,w^{-1}(l))$ partition the permutation array of $w$ into regions labeled $A-H$ as in Figure \ref{F:split_w}.  By definition of $m$ and $l$, the regions $D$ and $E$ must be empty.  Moreover, Lemma \ref{L:Parabolic perm_structure} part (1) and Lemma \ref{L:support_v} imply that
\begin{equation}\label{Eq:Supp_v}
S(v)=\{s_k \ |\ l\leq k<m\}.
\end{equation}
Similarly, the permutation array of $u$ partitions into regions $A'-H'$ as in Figure \ref{F:split_w}.  By Lemma \ref{L:Parabolic perm_structure} part (2), the nodes in each region labeled $A-H$ maintain the same relative order of those in $A'-H'$ respectively.  In particular, $(r,w^{-1}(m))$ and $(r+1,w^{-1}(l))$ are nodes in the permutation array of $u$.  Furthermore, since regions $D$ and $E$ are empty, the sizes of regions $A$ and $H$ are the same as the size of regions $A'$ and $H'$.

Now suppose $w$ avoids the patterns $3|12$ and $23|1$ with respect to position $r$.  Then regions $B,G$ must be empty and regions $C,F$ must be decreasing in the permutation array of $w$.  Thus regions $B',G'$ are empty and regions $C',F'$ are decreasing in the permutation array of $u$ (See Figure \ref{F:wu_condition}).  Now Lemma \ref{L:descent} and Equation \eqref{Eq:Supp_v} imply that $D_L(u)$ contains $S(v)\setminus\{s_r\}$ and hence $w=vu$ is a BP decomposition.

\begin{figure}[H]
\begin{tikzpicture}[scale=0.45]
\fill[lightgray] (0,0) rectangle (5,2);
\fill[lightgray] (10,10) rectangle (5,8);
\draw[thick] (0,0) -- (0,10) -- (10,10) -- (10,0) -- (0,0);
\draw[thick] (0,2) -- (10,2);
\draw[thick] (0,8) -- (10,8);
\draw[thick] (2,2) -- (2,8);
\draw[thick] (8,2) -- (8,8);
\draw[thick] (5,0) -- (5,10);
\fill (2,2) circle (7pt);
\fill (8,8) circle (7pt);
\draw (2.5,9) node {$A$};\draw (1,5) node {$B$}; \draw (2.5,1) node {$D$};
\draw (3.5,5) node {$C$};\draw (7.5,9) node {$E$};\draw (6.5,5) node {$F$};
\draw (7.5,1) node {$H$};\draw (9,5) node {$G$};
\draw (-1,2) node {$m$};\draw (-1,8) node {$l$};
\draw (5,-1) node {$r$};
\end{tikzpicture}
\qquad\qquad
\begin{tikzpicture}[scale=0.45]
\fill[lightgray] (0,0) rectangle (5,5);
\fill[lightgray] (10,10) rectangle (5,4);
\draw[thick] (0,0) -- (0,10) -- (10,10) -- (10,0) -- (0,0);
\draw[thick] (5,2) -- (10,2);
\draw[thick] (0,8) -- (5,8);
\draw[thick] (0,5) -- (5,5);
\draw[thick] (5,4) -- (10,4);
\draw[thick] (2,5) -- (2,8);
\draw[thick] (8,2) -- (8,4);
\draw[thick] (5,0) -- (5,10);
\fill (2,5) circle (7pt);
\fill (8,4) circle (7pt);
\draw (2.5,9) node {$A'$};\draw (1,6.5) node {$B'$}; \draw (2.5,2.5) node {$D'$};
\draw (3.5,6.5) node {$C'$};\draw (7.5,7) node {$E'$};\draw (6.5,3) node {$F'$};
\draw (7.5,1) node {$H'$};\draw (9,3) node {$G'$};
\draw (11.5,2) node {$m$};\draw (-1,8) node {$l$};
\draw (5,-1) node {$r$};\draw (-1,5) node {$r$};\draw (11.5,4) node {$r+1$};
\end{tikzpicture}
\caption{Permutation arrays of $w$ and $u$ partitioned by nodes at $(m,w^{-1}(m))$ and $(l,w^{-1}(l))$.\label{F:split_w}}
\end{figure}
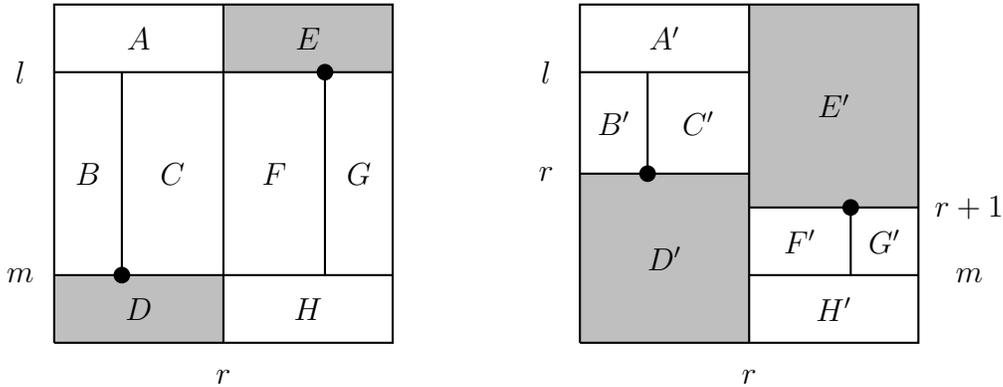

Conversely, suppose $S(v)\setminus\{s_r\}\subseteq D_L(u).$  In particular, Lemma \ref{L:descent} and Equation \eqref{Eq:Supp_v} say that $u^{-1}(k+1)<u^{-1}(k)$ for all $k\in [l,r-1]\sqcup [r+1,m-1]$.  This implies that regions $B',G'$ are empty and regions $C',F'$ are decreasing in the permutation array of $u$.  Hence regions $B,G$ are empty and regions $C,F$ are decreasing in the permutation array of $w$.  Thus $w$ avoids both split patterns $3|12$ and $23|1$ with respect to position $r$.  This completes the proof.

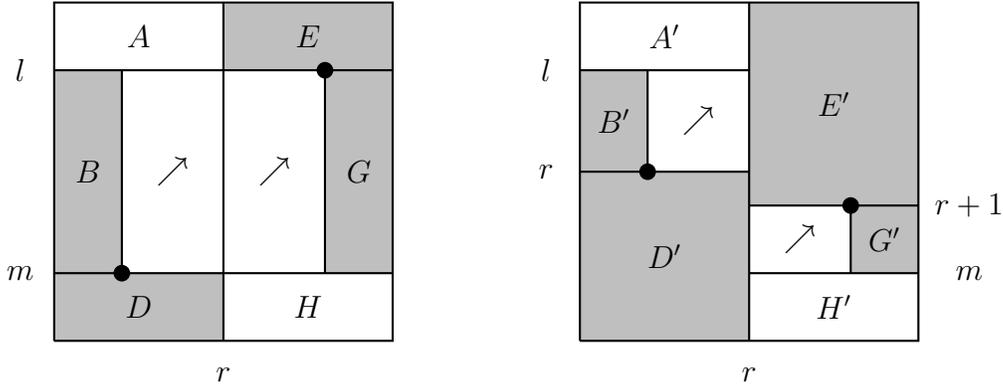
\begin{figure}[H]
\begin{tikzpicture}[scale=0.45]
\fill[lightgray] (0,0) rectangle (2,8);
\fill[lightgray] (0,0) rectangle (5,2);
\fill[lightgray] (10,10) rectangle (5,8);
\fill[lightgray] (10,10) rectangle (8,2);
\draw[thick] (0,0) -- (0,10) -- (10,10) -- (10,0) -- (0,0);
\draw[thick] (0,2) -- (10,2);
\draw[thick] (0,8) -- (10,8);
\draw[thick] (2,2) -- (2,8);
\draw[thick] (8,2) -- (8,8);
\draw[thick] (5,0) -- (5,10);
\fill (2,2) circle (7pt);
\fill (8,8) circle (7pt);
\draw (2.5,9) node {$A$};\draw (1,5) node {$B$}; \draw (2.5,1) node {$D$};
\draw (3.5,5) node {$\nearrow$};\draw (7.5,9) node {$E$};\draw (6.5,5) node {$\nearrow$};
\draw (7.5,1) node {$H$};\draw (9,5) node {$G$};
\draw (-1,2) node {$m$};\draw (-1,8) node {$l$};
\draw (5,-1) node {$r$};
\end{tikzpicture}
\qquad\qquad
\begin{tikzpicture}[scale=0.45]
\fill[lightgray] (0,0) rectangle (2,8);
\fill[lightgray] (0,0) rectangle (5,5);
\fill[lightgray] (10,10) rectangle (5,4);
\fill[lightgray] (10,10) rectangle (8,2);
\draw[thick] (0,0) -- (0,10) -- (10,10) -- (10,0) -- (0,0);
\draw[thick] (5,2) -- (10,2);
\draw[thick] (0,8) -- (5,8);
\draw[thick] (0,5) -- (5,5);
\draw[thick] (5,4) -- (10,4);
\draw[thick] (2,5) -- (2,8);
\draw[thick] (8,2) -- (8,4);
\draw[thick] (5,0) -- (5,10);
\fill (2,5) circle (7pt);
\fill (8,4) circle (7pt);
\draw (2.5,9) node {$A'$};\draw (1,6.5) node {$B'$}; \draw (2.5,2.5) node {$D'$};
\draw (3.5,6.5) node {$\nearrow$};\draw (7.5,7) node {$E'$};\draw (6.5,3) node {$\nearrow$};
\draw (7.5,1) node {$H'$};\draw (9,3) node {$G'$};
\draw (11.5,2) node {$m$};\draw (-1,8) node {$l$};
\draw (5,-1) node {$r$};\draw (-1,5) node {$r$};\draw (11.5,4) node {$r+1$};
\end{tikzpicture}
\caption{Permutation arrays of $w$ and $u$ with $w$ avoiding $3|12$ and $23|1$ with respect to position $r$ or equivalently, $S(v)\setminus\{s_r\}\subseteq D_L(u)$.\label{F:wu_condition}}
\end{figure}

\end{proof}

Before we prove Theorem \ref{T:main2}, we need the following proposition.

\begin{prop}\label{P:exists_BP}
If $w\in W$ avoids $3412$, $52341$ and $635241$, then there exists $r<n$ such that $w$ avoids $3|12$ and $23|1$ with respect to position $r$.  Furthermore, if $S(w)\neq \emptyset$, then we can choose $r$ such that $s_r\in S(w)$.
\end{prop}

\begin{proof}
We prove the first part of Proposition \ref{P:exists_BP} by contraction.  Suppose for every  position $r<n$, $w$ contains either $3|12$ or $23|1$.  In particular, $w$ contains $3|12$ with respect to position $r=1$.  Any $w(1)w(i)w(j)$ in relative position $3|12$ partitions the permutation array of $w$ into regions labelled $A-K$ as in Figure \ref{F:Case0}.  Moreover, we can choose nodes $(w(i),i),(w(j),j)$ such that regions $E,F, J$ are empty.  Since $w$ avoids 3412, region $D$ must also be empty and regions $C$ and $I$ must be decreasing.

\begin{figure}[h]
\begin{tikzpicture}[scale=0.45]
\draw[thick] (0,0) -- (0,10) -- (10,10) -- (10,0) -- (0,0);
\draw[thick] (0,1) -- (10,1);
\draw[thick] (0,4) -- (7,4);
\draw[thick] (0,7) -- (10,7);
\draw[thick] (0,0) -- (0,10);
\draw[thick] (3,0) -- (3,10);
\draw[thick] (7,0) -- (7,10);
\fill (0,1) circle (7pt);
\fill (3,7) circle (7pt);
\fill (7,4) circle (7pt);
\draw (1.5,8.5) node {$A$};\draw (1.5,5.5) node {$B$}; \draw (1.5,2.5) node {$C$};\draw (1.5,0.5) node {$D$};
\draw (5,8.5) node {$E$};\draw (5,5.5) node {$F$}; \draw (5,2.5) node {$G$};\draw (5,0.5) node {$H$};
\draw (8.5,8.5) node {$I$};\draw (8.5,4) node {$J$}; \draw (8.5,0.5) node {$K$};
\draw (0,-1) node {$1$};\draw (3,-1) node {$i$}; \draw (7,-1) node {$j$};
\draw (-1.5,1) node {$w(1)$};\draw (-1.5,7) node {$w(i)$}; \draw (-1.5,4) node {$w(j)$};
\end{tikzpicture}
\begin{tikzpicture}[scale=0.45]
\draw[thick,->](0,6)--(1,6);
\draw (-1,0) node {};\draw (2,0) node {};
\end{tikzpicture}
\begin{tikzpicture}[scale=0.45]
\fill[lightgray] (0,0) rectangle (3,1);
\fill[lightgray] (3,4) rectangle (7,10);
\fill[lightgray] (7,1) rectangle (10,7);
\draw[thick] (0,1) -- (10,1);
\draw[thick] (0,4) -- (7,4);
\draw[thick] (0,7) -- (10,7);
\draw[thick] (0,0) -- (0,10);
\draw[thick] (3,0) -- (3,10);
\draw[thick] (7,0) -- (7,10);
\draw[thick] (0,0) -- (0,10) -- (10,10) -- (10,0) -- (0,0);
\fill (0,1) circle (7pt);
\fill (3,7) circle (7pt);
\fill (7,4) circle (7pt);
\draw (1.5,8.5) node {$A$};\draw (1.5,5.5) node {$B$}; \draw (1.5,2.5) node {$\nearrow$};\draw (1.5,0.5) node {$D$};
\draw (5,8.5) node {$E$};\draw (5,5.5) node {$F$}; \draw (5,2.5) node {$G$};\draw (5,0.5) node {$H$};
\draw (8.5,8.5) node {$\nearrow$};\draw (8.5,4) node {$J$}; \draw (8.5,0.5) node {$K$};
\draw (0,-1) node {$1$};\draw (3,-1) node {$i$}; \draw (7,-1) node {$j$};
\draw (-1.5,1) node {$w(1)$};\draw (-1.5,7) node {$w(i)$}; \draw (-1.5,4) node {$w(j)$};
\end{tikzpicture}
\caption{Permutation array of $w$ containing $3|12$ with respect to position $r=1$.\label{F:Case0}}
\end{figure}
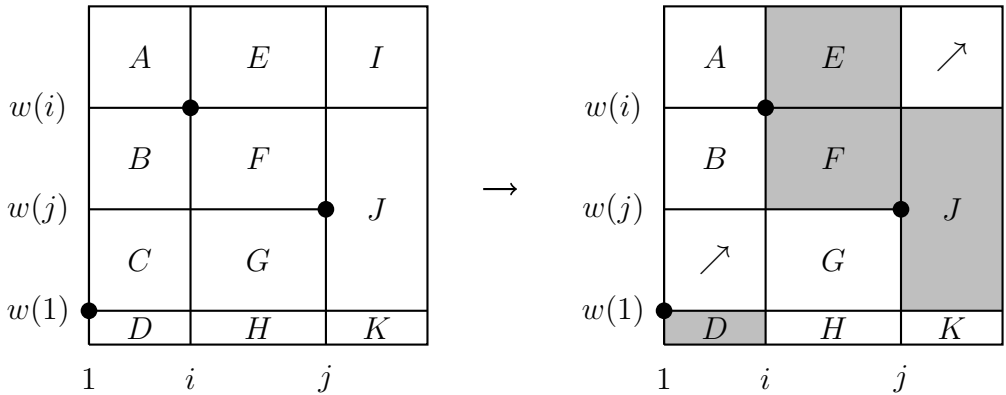

Now $w$ contains either pattern $3|12$ or $23|1$ with respect to position $r=i$.  We consider several cases depending on if region $I$ is empty or nonempty and if $w$ contains $3|12$ or $23|1$ with respect to position $i$.

\textbf{Case 1:}  Suppose the region $I$ is nonempty and $w$ contains $3|12$ with respect to position $i$.  Since regions $D,E,F$ and $J$ are empty and $I$ is decreasing, the permutation array of $w$ must contain two increasing nodes in region $G$ as in Figure \ref{F:Case1}.  This implies $w$ contains the pattern $52341$ which is a contradiction.

\medskip

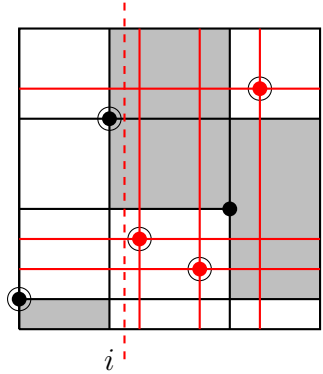
\begin{figure}[h]
\begin{tikzpicture}[scale=0.4]
\fill[lightgray] (0,0) rectangle (3,1);
\fill[lightgray] (3,4) rectangle (7,10);
\fill[lightgray] (7,1) rectangle (10,7);
\draw[thick] (0,1) -- (10,1);
\draw[thick] (0,4) -- (7,4);
\draw[thick] (0,7) -- (10,7);
\draw[thick] (0,0) -- (0,10);
\draw[thick] (3,0) -- (3,10);
\draw[thick] (7,0) -- (7,10);
\draw[thick] (0,0) -- (0,10) -- (10,10) -- (10,0) -- (0,0);
\fill (0,1) circle (7pt); \draw (0,1) circle (11pt);
\fill (3,7) circle (7pt); \draw (3,7) circle (11pt);
\fill (7,4) circle (7pt);
\fill[red] (4,3) circle (7pt);\fill[red] (6,2) circle (7pt);\fill[red] (8,8) circle (7pt);
\draw (4,3) circle (11pt);\draw (6,2) circle (11pt);\draw (8,8) circle (11pt);
\draw[thick,red] (4,0)--(4,10);\draw[thick,red](0,3)--(10,3);
\draw[thick,red] (6,0)--(6,10);\draw[thick,red](0,2)--(10,2);
\draw[thick,red] (8,0)--(8,10);\draw[thick,red](0,8)--(10,8);
\draw (3,-1) node {$i$};
\draw[dashed,thick, red] (3.5,-1)--(3.5,11);
\end{tikzpicture}
\caption{Permutation array of $w$ containing $3|12$ with respect to $r=i$ and region $I$ is nonempty.\label{F:Case1}}
\end{figure}

\textbf{Case 2:} Suppose the region $I$ is nonempty and $w$ contains $23|1$ with respect to position $i$.  If region $A$ has a node belonging to the pattern $23|1$, then $w$ contains the pattern $52341$.  Otherwise, since region $C$ is decreasing, $w$ must contain a pair of increasing nodes in region $B$ or $B\cup C.$  If the nodes are in region $B$, then $w$ contains $52341$ and if the nodes are in region $B\cup C$, then $w$ contains $635241$. See Figure \ref{F:Case2} for an illustration of these three subcases.

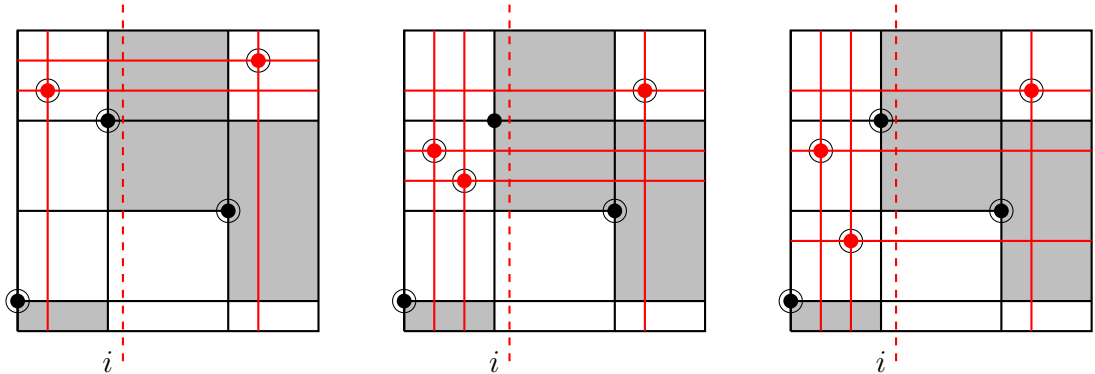
\begin{figure}[h]
\begin{tikzpicture}[scale=0.4]
\fill[lightgray] (0,0) rectangle (3,1);
\fill[lightgray] (3,4) rectangle (7,10);
\fill[lightgray] (7,1) rectangle (10,7);
\draw[thick] (0,1) -- (10,1);
\draw[thick] (0,4) -- (7,4);
\draw[thick] (0,7) -- (10,7);
\draw[thick] (0,0) -- (0,10);
\draw[thick] (3,0) -- (3,10);
\draw[thick] (7,0) -- (7,10);
\fill (0,1) circle (7pt);\draw (0,1) circle (11pt);
\fill (3,7) circle (7pt);\draw (3,7) circle (11pt);
\fill (7,4) circle (7pt);\draw (7,4) circle (11pt);
\draw[thick] (0,0) -- (0,10) -- (10,10) -- (10,0) -- (0,0);
\fill[red] (1,8) circle (7pt);\fill[red] (8,9) circle (7pt);
\draw (1,8) circle (11pt);\draw (8,9) circle (11pt);
\draw[thick,red] (1,0)--(1,10);\draw[thick,red](0,8)--(10,8);
\draw[thick,red] (8,0)--(8,10);\draw[thick,red](0,9)--(10,9);
\draw (3,-1) node {$i$};
\draw[dashed,thick, red] (3.5,-1)--(3.5,11);
\end{tikzpicture}
\qquad
\begin{tikzpicture}[scale=0.4]
\fill[lightgray] (0,0) rectangle (3,1);
\fill[lightgray] (3,4) rectangle (7,10);
\fill[lightgray] (7,1) rectangle (10,7);
\draw[thick] (0,1) -- (10,1);
\draw[thick] (0,4) -- (7,4);
\draw[thick] (0,7) -- (10,7);
\draw[thick] (0,0) -- (0,10);
\draw[thick] (3,0) -- (3,10);
\draw[thick] (7,0) -- (7,10);
\draw[thick] (0,0) -- (0,10) -- (10,10) -- (10,0) -- (0,0);
\fill (0,1) circle (7pt);\draw (0,1) circle (11pt);
\fill (3,7) circle (7pt);
\fill (7,4) circle (7pt);\draw (7,4) circle (11pt);
\fill[red] (1,6) circle (7pt);\fill[red] (2,5) circle (7pt);\fill[red] (8,8) circle (7pt);
\draw (1,6) circle (11pt);\draw (2,5) circle (11pt);\draw (8,8) circle (11pt);
\draw[thick,red] (1,0)--(1,10);\draw[thick,red](0,6)--(10,6);
\draw[thick,red] (2,0)--(2,10);\draw[thick,red](0,5)--(10,5);
\draw[thick,red] (8,0)--(8,10);\draw[thick,red](0,8)--(10,8);
\draw (3,-1) node {$i$};
\draw[dashed,thick, red] (3.5,-1)--(3.5,11);
\end{tikzpicture}
\qquad
\begin{tikzpicture}[scale=0.4]
\fill[lightgray] (0,0) rectangle (3,1);
\fill[lightgray] (3,4) rectangle (7,10);
\fill[lightgray] (7,1) rectangle (10,7);
\draw[thick] (0,1) -- (10,1);
\draw[thick] (0,4) -- (7,4);
\draw[thick] (0,7) -- (10,7);
\draw[thick] (0,0) -- (0,10);
\draw[thick] (3,0) -- (3,10);
\draw[thick] (7,0) -- (7,10);
\draw[thick] (0,0) -- (0,10) -- (10,10) -- (10,0) -- (0,0);
\fill (0,1) circle (7pt);\draw (0,1) circle (11pt);
\fill (3,7) circle (7pt);\draw (3,7) circle (11pt);
\fill (7,4) circle (7pt);\draw (7,4) circle (11pt);
\fill[red] (1,6) circle (7pt);\fill[red] (2,3) circle (7pt);\fill[red] (8,8) circle (7pt);
\draw (1,6) circle (11pt);\draw (2,3) circle (11pt);\draw (8,8) circle (11pt);
\draw[thick,red] (1,0)--(1,10);\draw[thick,red](0,6)--(10,6);
\draw[thick,red] (2,0)--(2,10);\draw[thick,red](0,3)--(10,3);
\draw[thick,red] (8,0)--(8,10);\draw[thick,red](0,8)--(10,8);
\draw (3,-1) node {$i$};
\draw[dashed,thick, red] (3.5,-1)--(3.5,11);
\end{tikzpicture}
\caption{Permutation array of $w$ containing $23|1$ with respect to $r=i$ and region $I$ is nonempty.\label{F:Case2}}
\end{figure}

\textbf{Case 3:} Suppose the region $I$ is empty.  Since region $C$ is decreasing, it is not possible for $w$ to contain $23|1$ with respect to position $i$.  Hence $w$ contains $3|12$ and thus region $G$ must contain a pair of increasing nodes.  These nodes partition region $G\cup H$ into sub-regions labeled $A'-K'$ as in Figure \ref{F:Case3}.  Choose increasing nodes $(w(i'),i')$ and $(w(j'),j')$ in region $G$, so that regions $E',F'$ and $J'$ are empty.  Also, since $w$ avoids 3412 and 52341, we can further assume that regions $A'$ and $D'$ are empty and that regions $C'$ and $I'$ are decreasing.

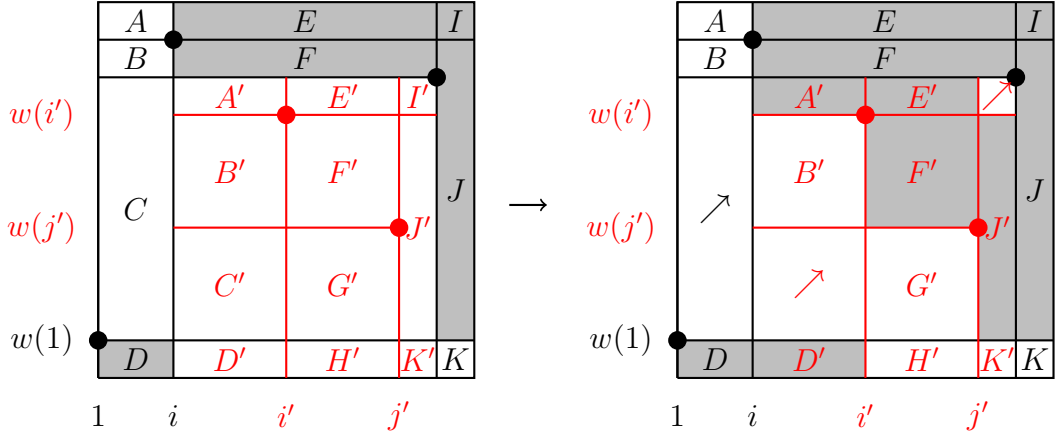
\begin{figure}[h]
$$\begin{tikzpicture}[scale=0.5]
\fill[lightgray] (2,9) rectangle (10,10);
\fill[lightgray] (2,8) rectangle (9,9);
\fill[lightgray] (9,1) rectangle (10,9);
\fill[lightgray] (0,0) rectangle (2,1);
\draw[thick] (0,1) -- (10,1);
\draw[thick] (0,8) -- (9,8);
\draw[thick] (0,9) -- (10,9);
\draw[thick] (0,0) -- (0,10);
\draw[thick] (2,0) -- (2,10);
\draw[thick] (9,0) -- (9,10);
\draw[thick] (0,0) -- (0,10) -- (10,10) -- (10,0) -- (0,0);
\fill (0,1) circle (7pt);
\fill (2,9) circle (7pt);
\fill (9,8) circle (7pt);
\draw (1,9.5) node {$A$};\draw (1,8.5) node {$B$}; \draw (1,4.5) node {$C$};\draw (1,0.5) node {$D$};
\draw (5.5,9.5) node {$E$};\draw (5.5,8.5) node {$F$}; \draw (5.5,4.5) node {};\draw (5.5,0.5) node {};
\draw (9.5,9.5) node {$I$};\draw (9.5,5) node {$J$}; \draw (9.5,0.5) node {$K$};
\fill[red] (5,7) circle (7pt);\fill[red] (8,4) circle (7pt);
\draw[thick,red] (5,0)--(5,8);\draw[thick,red](2,7)--(9,7);
\draw[thick,red] (8,0)--(8,8);\draw[thick,red](2,4)--(8,4);
\draw[red] (3.5,7.5) node {$A'$};\draw[red] (3.5,5.5) node {$B'$}; \draw[red] (3.5,2.5) node {$C'$};\draw[red] (3.5,0.5) node {$D'$};
\draw[red] (6.5,7.5) node {$E'$};\draw[red] (6.5,5.5) node {$F'$}; \draw[red] (6.5,2.5) node {$G'$};\draw[red] (6.5,0.5) node {$H'$};
\draw[red] (8.5,7.5) node {$I'$};\draw[red] (8.5,4) node {$J'$}; \draw[red] (8.5,0.5) node {$K'$};
\draw (0,-1) node {$1$};\draw (2,-1) node {$i$}; \draw[red] (5,-1) node {$i'$}; \draw[red] (8,-1) node {$j'$};
\draw (-1.5,1) node {$w(1)$};\draw[red] (-1.5,7) node {$w(i')$}; \draw[red] (-1.5,4) node {$w(j')$};
\end{tikzpicture}
\begin{tikzpicture}[scale=0.5]
\draw[thick,->](0,6)--(1,6);
\draw (-0.5,0) node {};\draw (1.5,0) node {};
\end{tikzpicture}
\begin{tikzpicture}[scale=0.5]
\fill[lightgray] (2,9) rectangle (10,10);
\fill[lightgray] (2,8) rectangle (9,9);
\fill[lightgray] (2,7) rectangle (5,8);
\fill[lightgray] (9,1) rectangle (10,9);
\fill[lightgray] (0,0) rectangle (5,1);
\fill[lightgray] (8,1) rectangle (9,7);
\fill[lightgray] (5,4) rectangle (8,8);
\draw[thick] (0,1) -- (10,1);
\draw[thick] (0,8) -- (9,8);
\draw[thick] (0,9) -- (10,9);
\draw[thick] (0,0) -- (0,10);
\draw[thick] (2,0) -- (2,10);
\draw[thick] (9,0) -- (9,10);
\draw[thick] (0,0) -- (0,10) -- (10,10) -- (10,0) -- (0,0);
\fill (0,1) circle (7pt);
\fill (2,9) circle (7pt);
\fill (9,8) circle (7pt);
\draw (1,9.5) node {$A$};\draw (1,8.5) node {$B$}; \draw (1,4.5) node {$\nearrow$};\draw (1,0.5) node {$D$};
\draw (5.5,9.5) node {$E$};\draw (5.5,8.5) node {$F$}; \draw (5.5,4.5) node {};\draw (5.5,0.5) node {};
\draw (9.5,9.5) node {$I$};\draw (9.5,5) node {$J$}; \draw (9.5,0.5) node {$K$};
\fill[red] (5,7) circle (7pt);\fill[red] (8,4) circle (7pt);
\draw[thick,red] (5,0)--(5,8);\draw[thick,red](2,7)--(9,7);
\draw[thick,red] (8,0)--(8,8);\draw[thick,red](2,4)--(8,4);
\draw[red] (3.5,7.5) node {$A'$};\draw[red] (3.5,5.5) node {$B'$}; \draw[red] (3.5,2.5) node {$\nearrow$};\draw[red] (3.5,0.5) node {$D'$};
\draw[red] (6.5,7.5) node {$E'$};\draw[red] (6.5,5.5) node {$F'$}; \draw[red] (6.5,2.5) node {$G'$};\draw[red] (6.5,0.5) node {$H'$};
\draw[red] (8.5,7.5) node {$\nearrow$};\draw[red] (8.5,4) node {$J'$}; \draw[red] (8.5,0.5) node {$K'$};
\draw (0,-1) node {$1$};\draw (2,-1) node {$i$}; \draw[red] (5,-1) node {$i'$}; \draw[red] (8,-1) node {$j'$};
\draw (-1.5,1) node {$w(1)$};\draw[red] (-1.5,7) node {$w(i')$}; \draw[red] (-1.5,4) node {$w(j')$};
\end{tikzpicture}$$
\caption{Permutation array of $w$ containing $3|12$ with respect to position $r=i$ and region $I$ is empty.\label{F:Case3}}
\end{figure}

Now $w$ contains $3|12$ or $23|1$ with respect to position $r=i'$. First, if $w$ contains $3|12$, then, since region $I'$ is decreasing, $w$ must have a pair of increasing nodes in region $G'$.  This implies $w$ contains 52341.

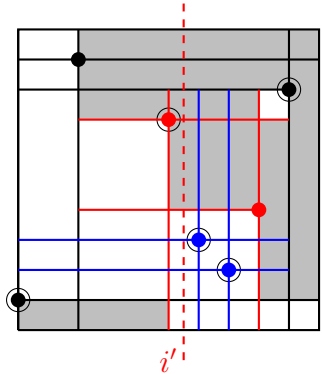
\begin{figure}[h]
\begin{tikzpicture}[scale=0.4]
\fill[lightgray] (2,9) rectangle (10,10);
\fill[lightgray] (2,8) rectangle (9,9);
\fill[lightgray] (9,1) rectangle (10,9);
\fill[lightgray] (0,0) rectangle (5,1);
\fill[lightgray] (8,1) rectangle (9,7);
\fill[lightgray] (5,4) rectangle (8,8);
\fill[lightgray] (2,7) rectangle (5,8);
\draw[thick] (0,1) -- (10,1);
\draw[thick] (0,8) -- (9,8);
\draw[thick] (0,9) -- (10,9);
\draw[thick] (0,0) -- (0,10);
\draw[thick] (2,0) -- (2,10);
\draw[thick] (9,0) -- (9,10);
\draw[thick] (0,0) -- (0,10) -- (10,10) -- (10,0) -- (0,0);
\fill (0,1) circle (7pt);\draw (0,1) circle (11pt);
\fill (2,9) circle (7pt);
\fill (9,8) circle (7pt);\draw (9,8) circle (11pt);
\fill[red] (5,7) circle (7pt);\fill[red] (8,4) circle (7pt);
\draw (5,7) circle (11pt);
\draw[thick,red] (5,0)--(5,8);\draw[thick,red](2,7)--(9,7);
\draw[thick,red] (8,0)--(8,8);\draw[thick,red](2,4)--(8,4);
\fill[blue] (6,3) circle (7pt);\fill[blue] (7,2) circle (7pt);
\draw (6,3) circle (11pt);\draw (7,2) circle (11pt);
\draw[thick,blue] (6,0)--(6,8);\draw[thick,blue](0,3)--(9,3);
\draw[thick,blue] (7,0)--(7,8);\draw[thick,blue](0,2)--(9,2);
\draw[red] (5,-1) node {$i'$};
\draw[dashed,thick, red] (5.5,-1)--(5.5,11);
\end{tikzpicture}
\caption{Permutation array of $w$ containing $3|12$ with respect to position $r=i'$.\label{F:Case3a}}
\end{figure}

If $w$ contains $23|1$, then the fact that regions $C$ and $C'$ are decreasing implies that $w$ has a pair of increasing nodes in either regions $B', B'\cup C', C\cup B'$ or $C\cup C'.$  If $w$ contains increasing nodes in regions $B'$ or $B'\cup C'$, then $w$ contains $52341$ or $635241$ respectively as in Figure \ref{F:Case3b}.

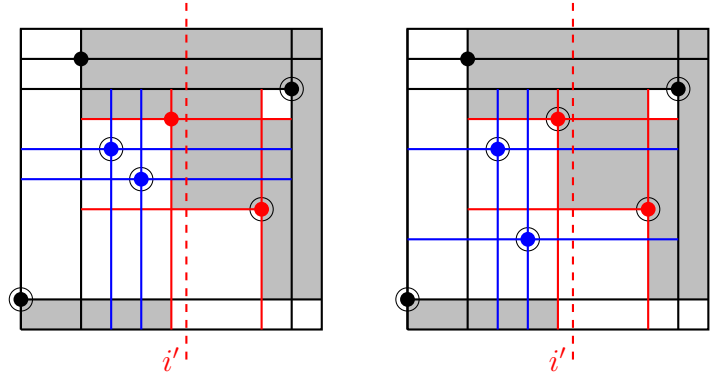
\begin{figure}[h]
\begin{tikzpicture}[scale=0.4]
\fill[lightgray] (2,9) rectangle (10,10);
\fill[lightgray] (2,8) rectangle (9,9);
\fill[lightgray] (9,1) rectangle (10,9);
\fill[lightgray] (0,0) rectangle (5,1);
\fill[lightgray] (8,1) rectangle (9,7);
\fill[lightgray] (5,4) rectangle (8,8);
\fill[lightgray] (2,7) rectangle (5,8);
\draw[thick] (0,1) -- (10,1);
\draw[thick] (0,8) -- (9,8);
\draw[thick] (0,9) -- (10,9);
\draw[thick] (0,0) -- (0,10);
\draw[thick] (2,0) -- (2,10);
\draw[thick] (9,0) -- (9,10);
\draw[thick] (0,0) -- (0,10) -- (10,10) -- (10,0) -- (0,0);
\fill (0,1) circle (7pt);\draw (0,1) circle (11pt);
\fill (2,9) circle (7pt);
\fill (9,8) circle (7pt);\draw (9,8) circle (11pt);
\fill[red] (5,7) circle (7pt);\fill[red] (8,4) circle (7pt);
\draw (8,4) circle (11pt);
\draw[thick,red] (5,0)--(5,8);\draw[thick,red](2,7)--(9,7);
\draw[thick,red] (8,0)--(8,8);\draw[thick,red](2,4)--(8,4);
\fill[blue] (3,6) circle (7pt);\fill[blue] (4,5) circle (7pt);
\draw (3,6) circle (11pt);\draw (4,5) circle (11pt);
\draw[thick,blue] (3,0)--(3,8);\draw[thick,blue](0,6)--(9,6);
\draw[thick,blue] (4,0)--(4,8);\draw[thick,blue](0,5)--(9,5);
\draw[red] (5,-1) node {$i'$};
\draw[dashed,thick, red] (5.5,-1)--(5.5,11);
\end{tikzpicture}
\qquad
\begin{tikzpicture}[scale=0.4]
\fill[lightgray] (2,9) rectangle (10,10);
\fill[lightgray] (2,8) rectangle (9,9);
\fill[lightgray] (9,1) rectangle (10,9);
\fill[lightgray] (0,0) rectangle (5,1);
\fill[lightgray] (8,1) rectangle (9,7);
\fill[lightgray] (5,4) rectangle (8,8);
\fill[lightgray] (2,7) rectangle (5,8);
\draw[thick] (0,1) -- (10,1);
\draw[thick] (0,8) -- (9,8);
\draw[thick] (0,9) -- (10,9);
\draw[thick] (0,0) -- (0,10);
\draw[thick] (2,0) -- (2,10);
\draw[thick] (9,0) -- (9,10);
\fill (0,1) circle (7pt);\draw (0,1) circle (11pt);
\fill (2,9) circle (7pt);
\fill (9,8) circle (7pt);\draw (9,8) circle (11pt);
\draw[thick] (0,0) -- (0,10) -- (10,10) -- (10,0) -- (0,0);
\fill[red] (5,7) circle (7pt);\fill[red] (8,4) circle (7pt);
\draw (5,7) circle (11pt);\draw (8,4) circle (11pt);
\draw[thick,red] (5,0)--(5,8);\draw[thick,red](2,7)--(9,7);
\draw[thick,red] (8,0)--(8,8);\draw[thick,red](2,4)--(8,4);
\fill[blue] (3,6) circle (7pt);\fill[blue] (4,3) circle (7pt);
\draw (3,6) circle (11pt);\draw (4,3) circle (11pt);
\draw[thick,blue] (3,0)--(3,8);\draw[thick,blue](0,6)--(9,6);
\draw[thick,blue] (4,0)--(4,8);\draw[thick,blue](0,3)--(9,3);
\draw[red] (5,-1) node {$i'$};
\draw[dashed,thick, red] (5.5,-1)--(5.5,11);
\end{tikzpicture}
\caption{Permutation array of $w$ containing $23|1$ with respect to position $r=i'$ using regions $B'$ and $B'\cup C'$.\label{F:Case3b}}
\end{figure}

Finally, if $w$ contains increasing nodes in regions $C\cup B'$ or $C\cup C'$, then we have the following three possibilities as in Figure \ref{F:Case3c}.

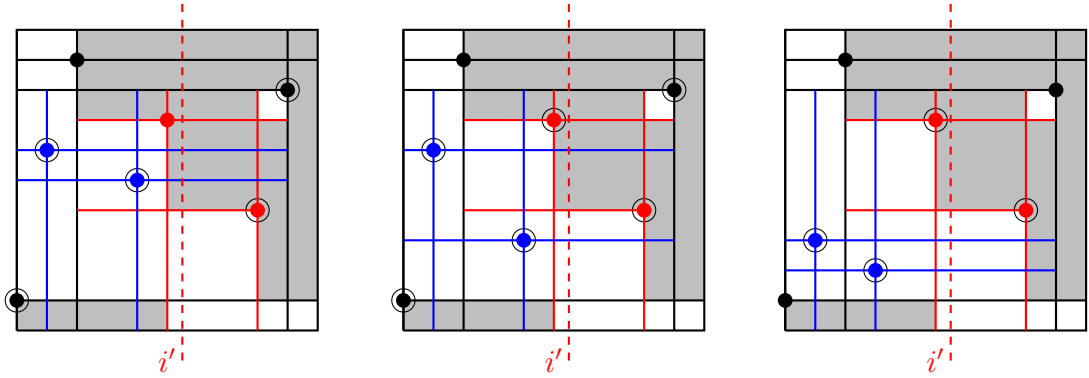
\begin{figure}[h]
\begin{tikzpicture}[scale=0.4]
\fill[lightgray] (2,9) rectangle (10,10);
\fill[lightgray] (2,8) rectangle (9,9);
\fill[lightgray] (9,1) rectangle (10,9);
\fill[lightgray] (0,0) rectangle (5,1);
\fill[lightgray] (8,1) rectangle (9,7);
\fill[lightgray] (5,4) rectangle (8,8);
\fill[lightgray] (2,7) rectangle (5,8);
\draw[thick] (0,1) -- (10,1);
\draw[thick] (0,8) -- (9,8);
\draw[thick] (0,9) -- (10,9);
\draw[thick] (0,0) -- (0,10);
\draw[thick] (2,0) -- (2,10);
\draw[thick] (9,0) -- (9,10);
\draw[thick] (0,0) -- (0,10) -- (10,10) -- (10,0) -- (0,0);
\fill (0,1) circle (7pt);\draw (0,1) circle (11pt);
\fill (2,9) circle (7pt);
\fill (9,8) circle (7pt);\draw (9,8) circle (11pt);
\fill[red] (5,7) circle (7pt);\fill[red] (8,4) circle (7pt);
\draw (8,4) circle (11pt);
\draw[thick,red] (5,0)--(5,8);\draw[thick,red](2,7)--(9,7);
\draw[thick,red] (8,0)--(8,8);\draw[thick,red](2,4)--(8,4);
\fill[blue] (1,6) circle (7pt);\fill[blue] (4,5) circle (7pt);
\draw (1,6) circle (11pt);\draw (4,5) circle (11pt);
\draw[thick,blue] (1,0)--(1,8);\draw[thick,blue](0,6)--(9,6);
\draw[thick,blue] (4,0)--(4,8);\draw[thick,blue](0,5)--(9,5);
\draw[red] (5,-1) node {$i'$};
\draw[dashed,thick, red] (5.5,-1)--(5.5,11);
\end{tikzpicture}
\qquad
\begin{tikzpicture}[scale=0.4]
\fill[lightgray] (2,9) rectangle (10,10);
\fill[lightgray] (2,8) rectangle (9,9);
\fill[lightgray] (9,1) rectangle (10,9);
\fill[lightgray] (0,0) rectangle (5,1);
\fill[lightgray] (8,1) rectangle (9,7);
\fill[lightgray] (5,4) rectangle (8,8);
\fill[lightgray] (2,7) rectangle (5,8);
\draw[thick] (0,1) -- (10,1);
\draw[thick] (0,8) -- (9,8);
\draw[thick] (0,9) -- (10,9);
\draw[thick] (0,0) -- (0,10);
\draw[thick] (2,0) -- (2,10);
\draw[thick] (9,0) -- (9,10);
\draw[thick] (0,0) -- (0,10) -- (10,10) -- (10,0) -- (0,0);
\fill (0,1) circle (7pt);\draw (0,1) circle (11pt);
\fill (2,9) circle (7pt);
\fill (9,8) circle (7pt);\draw (9,8) circle (11pt);
\fill[red] (5,7) circle (7pt);\fill[red] (8,4) circle (7pt);
\draw (5,7) circle (11pt);\draw (8,4) circle (11pt);
\draw[thick,red] (5,0)--(5,8);\draw[thick,red](2,7)--(9,7);
\draw[thick,red] (8,0)--(8,8);\draw[thick,red](2,4)--(8,4);
\fill[blue] (1,6) circle (7pt);\fill[blue] (4,3) circle (7pt);
\draw (1,6) circle (11pt);\draw (4,3) circle (11pt);
\draw[thick,blue] (1,0)--(1,8);\draw[thick,blue](0,6)--(9,6);
\draw[thick,blue] (4,0)--(4,8);\draw[thick,blue](0,3)--(9,3);
\draw[red] (5,-1) node {$i'$};
\draw[dashed,thick, red] (5.5,-1)--(5.5,11);
\end{tikzpicture}
\qquad
\begin{tikzpicture}[scale=0.4]
\fill[lightgray] (2,9) rectangle (10,10);
\fill[lightgray] (2,8) rectangle (9,9);
\fill[lightgray] (9,1) rectangle (10,9);
\fill[lightgray] (0,0) rectangle (5,1);
\fill[lightgray] (8,1) rectangle (9,7);
\fill[lightgray] (5,4) rectangle (8,8);
\fill[lightgray] (2,7) rectangle (5,8);
\draw[thick] (0,1) -- (10,1);
\draw[thick] (0,8) -- (9,8);
\draw[thick] (0,9) -- (10,9);
\draw[thick] (0,0) -- (0,10);
\draw[thick] (2,0) -- (2,10);
\draw[thick] (9,0) -- (9,10);
\draw[thick] (0,0) -- (0,10) -- (10,10) -- (10,0) -- (0,0);
\fill (0,1) circle (7pt);
\fill (2,9) circle (7pt);
\fill (9,8) circle (7pt);
\fill[red] (5,7) circle (7pt);\fill[red] (8,4) circle (7pt);
\draw (5,7) circle (11pt);\draw (8,4) circle (11pt);
\draw[thick,red] (5,0)--(5,8);\draw[thick,red](2,7)--(9,7);
\draw[thick,red] (8,0)--(8,8);\draw[thick,red](2,4)--(8,4);
\fill[blue] (1,3) circle (7pt);\fill[blue] (3,2) circle (7pt);
\draw (1,3) circle (11pt);\draw (3,2) circle (11pt);
\draw[thick,blue] (1,0)--(1,8);\draw[thick,blue](0,3)--(9,3);
\draw[thick,blue] (3,0)--(3,8);\draw[thick,blue](0,2)--(9,2);
\draw[red] (5,-1) node {$i'$};
\draw[dashed,thick, red] (5.5,-1)--(5.5,11);
\end{tikzpicture}
\caption{Permutation array of $w$ containing $23|1$ with respect to position $r=i'$ using regions $C\cup B'$ and $C\cup C'$.\label{F:Case3c}}
\end{figure}

We can see that $w$ contains 52341, 635241 and 3412 respectively for each of these possibilities.  This completes the first part of the proof.

For the second part, we show that position $r$ can be chosen so that $s_r\in S(w)$.  We proceed by induction on the size of the permutation $n$.  First note that if $n=2$, then the Proposition is true for $w=s_1$ and $r=1$.  Now if $w\in W=\mfS_n$ avoids the patterns $3412$, $52341$ and $635241$, then there exists $r<n$ where the parabolic decomposition $w=vu$ with respect to $J=S\setminus\{s_r\}$ is a BP decomposition.  If $s_r\in S(w)$, then we are done.  Otherwise, $s_r\notin S(w)$ which implies $w=u$.  Write
$$w=w_1|w_2=w(1)\cdots w(r)|w(r+1)\cdots w(n).$$
If $J_1=\{s_1,\cdots,s_{r-1}\}$ and $J_2=J\setminus J_1$, then Lemma \ref{L:Parabolic perm_structure} implies that $w_1$ and $w_2$ also avoid $3412$, $52341$ and $635241$ as permutations in $W_{J_1}\simeq\mfS_r$ and $W_{J_2}\simeq\mfS_{n-r}$ respectively.  Since either $r$ or $n-r$ is greater than $1$ we will assume, without loss of generality, that $r>1$ and $S(w_1)\neq \emptyset$.  By induction, there exists $r'<r$ for which $s_{r'}\in S(w_1)$ and $w_1$ avoids $3|12$ and $23|1$ with respect to position $r'$.  Now Lemma \ref{L:Parabolic perm_structure} implies $w$ also avoids $3|12$ and $23|1$ with respect to position $r'$.  But $S(w_1)\subseteq S(w)$ and hence $s_{r'}\in S(w)$.  This completes the proof.
\end{proof}

\begin{proof}[Proof of Theorem \ref{T:main2}]
By Proposition \ref{P:BP_complete}, it suffices to show that $w\in W$ avoids the patterns $3412$, $52341$ and $635241$ if and only if $w$ has a complete BP decomposition.  First assume the $w\in W$ avoids the patterns $3412$, $52341$ and $635241$.  We will show $w$ has a complete BP decomposition by induction on $\ell(w)$.  First note that if $w=e$, then the theorem is vacuously true.  If $w\neq e$, then by Theorem \ref{T:main} and Proposition \ref{P:exists_BP}, there exists $r<n$ such that $s_r\in S(w)$ and the parabolic decomposition $w=vu$ with respect to $J=S\setminus\{s_r\}$ is a BP decomposition.  Lemma \ref{L:Parabolic perm_structure} implies that $u$ also avoids the patterns $3412$, $52341$ and $635241$.  Since $s_r\in S(w)$, we must have $\ell(u)<\ell(w)$ and by induction, we are done.

Conversely, assume $w$ has a complete BP decomposition.  In particular, there exists $r<n$ such that the parabolic decomposition $w=vu$ with respect to $J=S\setminus\{s_r\}$ is a BP decomposition.  Note that if $s_r\notin S(w)$, then $w=u$.  Without loss of generality, we can assume $s_r\in S(w)$ and hence $\ell(u)<\ell(w)$.  By induction on $\ell(w)$, it suffices to show that if $u$ avoids $3412$, $52341$ and $635241$, then $w$ avoids those same patterns.  Write $w=w_1|w_2$ and $u=u_1|u_2$ with respect to position $r$ as in Lemma \ref{L:Parabolic perm_structure}.  Since $u$ avoids $3412$, $52341$ and $635241$, each subword $w_1$ and $w_2$ must also avoid these patterns.  Hence if $w$ contains one of $3412$, $52341$ or $635241$, it must contain this pattern using entries in both $w_1$ and $w_2$.  But then $w$ must contain either $3|12$ or $23|1$ with respect to position $r$.  Theorem \ref{T:main} implies that $w$ cannot have a BP decomposition with respect to $J$ which is a contradiction.  This completes the proof.\end{proof}

\section{Examples}

In this section we illustrate Theorems \ref{T:main} and \ref{T:main2} with several examples and write out the corresponding fiber bundle structures (or lack thereof) geometrically.  In each of these examples, we fix a flag $E_\bull:=E_1\subset E_2\subset E_3\subset \K^4$ where $\dim(E_i)=i$.

\begin{example}\label{Ex:1}
\emph{The Schubert variety indexed by $w=3241$ is
$$X_{3241}=\{V_\bull=(V_1\subset V_2\subset V_3\subset \K^4) \ |\ V_2\subset E_3\}.$$
It is easy to check that $w$ avoids the split patterns $3|12$ and $23|1$ with respect to positions $r=1,2$, but contains $23|1$ with respect to $r=3$.   Theorem \ref{T:main} implies that the projections maps $\pi_1,\pi_2$ are fiber bundle maps on $X_w$ while $\pi_3$ is not.  In particular, we have fibers
\begin{align*}
\pi_1^{-1}({\color{red} V_1})&=\{(V_2\subset V_3)\ |\ {\color{red} V_1}\subset V_2\subset E_3\}\simeq X_{1342}\\
\pi_2^{-1}({\color{red} V_2})&=\{(V_1\subset V_3)\ |\ V_1\subset {\color{red} V_2}\subset V_3\}\simeq \Fl(2)\times\Fl(2)\\
\pi_3^{-1}({\color{red} V_3})&=\{(V_1\subset V_2)\ |\ V_2\subset E_3\cap {\color{red} V_3}\}\simeq
\begin{cases}
\Fl(2) & \text{if}\quad \dim({\color{red} V_3}\cap E_3)=2\\
\Fl(3) & \text{if}\quad {\color{red} V_3}=E_3.
\end{cases}
\end{align*}
Theorem \ref{T:main2} implies that $X_w$ has a complete parabolic bundle structure.  Indeed, the permutations $231, 213$ and $123$ all induce complete parabolic bundle structures on $X_w$ while permutations 312, 321 and 123 do not.  For example, if $\sigma=231$ then the corresponding complete BP decomposition of $w$ in accordance to Proposition \ref{P:BP_complete} is $w=(s_1s_2)(s_3)(s_1)$
and the complete parabolic bundle structure on $X_w$ is
\begin{equation*}\label{Eq:Flag_bundle_map_Schubert ex}
X_w\overset{\pi^{\{1,2,3\}}_{\{2,3\}}}{\twoheadrightarrow} X_{2}\overset{\pi^{\{2,3\}}_{\{2\}}}{\twoheadrightarrow} X_1\twoheadrightarrow pt
\end{equation*}
where $X_2=\{V^{\{2,3\}}_\bull \ |\ V_2\subset E_3\}$ and $X_1=\{V^{\{2\}}_\bull \ |\ V_2\subset E_3\}.$
The fibers of each of the maps $\pi^{\{1,2,3\}}_{\{2,3\}}$ and $\pi^{\{2,3\}}_{\{2\}}$ are both isomorphic to $\Fl(2)$ while $X_1$ is isomorphic to $\Gr(2,3)$.}
\end{example}

\begin{example}\label{Ex:2}
\emph{The Schubert variety indexed by $w=4231$ is
$$X_{4231}=\{V_\bull=(V_1\subset V_2\subset V_3\subset \K^4) \ |\ \dim(V_2\cap E_2)\geq 1\}.$$
In this case $w$ contains $3|12$ with respect to $r=1$, contains $23|1$ with respect $r=3$ and avoids both split patterns with respect to $r=2$.  Hence $\pi_1,\pi_3$ are not fiber bundle maps while $\pi_2$ is a fiber bundle map on $X_{4231}$.  In particular,
\begin{align*}
\pi_1^{-1}({\color{red} V_1})&=\{(V_2\subset V_3)\ |\ \dim(V_2\cap E_2)\geq 1\}\simeq
\begin{cases}
X_{1342} & \text{if}\quad \dim({\color{red} V_1}\cap E_2)=0\\
\Fl(3) & \text{if}\quad {\color{red} V_1}\subset E_2.
\end{cases}\\
\pi_2^{-1}({\color{red} V_2})&=\{(V_1\subset V_3)\ |\ V_1\subset {\color{red} V_2}\subset V_3\}\simeq \Fl(2)\times\Fl(2)\\
\pi_3^{-1}({\color{red} V_3})&=\{(V_1\subset V_2)\ |\ \dim(V_2\cap E_2)\geq 1\}\simeq
\begin{cases}
X_{3124} & \text{if}\quad \dim({\color{red} V_3}\cap E_2)=1\\
\Fl(3) & \text{if}\quad E_2\subset{\color{red} V_3}.
\end{cases}
\end{align*}
Again, Theorem \ref{T:main2} implies that $X_w$ has a complete parabolic bundle structure.  For example, if $\sigma=213$, then the complete parabolic bundle structure on $X_w$ is
\begin{equation*}\label{Eq:Flag_bundle_map_Schubert ex}
X_w\overset{\pi^{\{1,2,3\}}_{\{1,2\}}}{\twoheadrightarrow} X_{2}\overset{\pi^{\{1,2\}}_{\{2\}}}{\twoheadrightarrow} X_1\twoheadrightarrow pt
\end{equation*}
where $X_2=\{V^{\{1,2\}}_\bull \ |\ \dim(V_2\cap E_2)\geq 1\}$ and $X_1=\{V^{\{2\}}_\bull \ |\ \dim(V_2\cap E_2)\geq 1\}.$
The fibers of each of the maps $\pi^{\{1,2,3\}}_{\{1,2\}}$ and $\pi^{\{1,2\}}_{\{2\}}$ are both isomorphic to $\Fl(2)$ while $X_1$ is isomorphic to $X^{\{s_1,s_3\}}_{2413}\subset \Gr(2,4)$.  The corresponding complete BP decomposition of $w$ is $w=(s_1s_3s_2)(s_1)(s_3)$.}
\end{example}

\begin{example}\label{Ex:3}
\emph{The Schubert variety indexed by $w=3412$ is
$$X_{3412}=\{V_\bull=(V_1\subset V_2\subset V_3\subset \K^4) \ |\ V_1\subset E_3,\ E_1\subset V_3\}.$$
In this case $w$ contains either $3|12$ or $23|1$ with respect to every $r$.  Hence all projections $\pi_r$ are not fiber bundle maps on $X_{3412}$.  In particular,
\begin{align*}
\pi_1^{-1}({\color{red} V_1})&=\{(V_2\subset V_3)\ |\ E_1\subset V_3\}\simeq
\begin{cases}
X_{1342} & \text{if}\quad \dim({\color{red} V_1}\cap E_1)=0\\
\Fl(3) & \text{if}\quad {\color{red} V_1}= E_1.
\end{cases}\\
\pi_2^{-1}({\color{red} V_2})&\simeq
\begin{cases}
X_{1234} & \text{if}\quad \dim({\color{red} V_2}\cap E_3)=1,\ \dim({\color{red} V_2}\cap E_1)=0\\
X_{1243} & \text{if}\quad E_1\subset {\color{red} V_2},\ \dim({\color{red} V_2}\cap E_3)=1\\
X_{2134} & \text{if}\quad {\color{red} V_2}\subset E_3,\ \dim({\color{red} V_2}\cap E_1)=0\\
\Fl(2)\times\Fl(2) & \text{if}\quad  E_1\subset{\color{red} V_2}\subset E_3\\
\end{cases}\\
\pi_3^{-1}({\color{red} V_3})&=\{(V_1\subset V_2)\ |\ V_1\subset E_3\}\simeq
\begin{cases}
X_{3124} & \text{if}\quad \dim({\color{red} V_3}\cap E_3)=2\\
\Fl(3) & \text{if}\quad {\color{red} V_3}=E_3.
\end{cases}
\end{align*}
Theorem \ref{T:main2} implies that $X_{3412}$ has no complete parabolic bundle structure.  It is easy to check that $\pi^{\{1,2,3\}}_{\{2,3\}}$ and $\pi^{\{1,2,3\}}_{\{1,2\}}$ are not fiber bundle maps, however $\pi^{\{1,2,3\}}_{\{1,3\}}$ is a fiber bundle map with fiber isomorphic to $\Fl(2)$.  In an attempt to continue the iterated fiber bundle, we see that neither of the projection maps $\pi^{\{1,3\}}_{\{1\}}, \pi^{\{1,3\}}_{\{3\}}$ induce fiber bundle structures on $\pi^{\{1,2,3\}}_{\{1,3\}}(X_{3412})$.  Hence, we confirm that $X_{3412}$ has no complete parabolic bundle structure.}
\end{example}

\bibliographystyle{amsalpha}
\bibliography{palindromic}

\end{document}